\documentclass[times]{elsarticle}

\makeatletter
\def\ps@pprintTitle{%
\let\@oddhead\@empty
\let\@evenhead\@empty
\def\@oddfoot{\centerline{\thepage}}%
\let\@evenfoot\@oddfoot}
\makeatother

\usepackage{amsthm,amsfonts}
\usepackage{mathrsfs}

\newtheorem{thm}{Theorem}
\newtheorem{lem}[thm]{Lemma}
\newtheorem{prop}[thm]{Proposition}
\theoremstyle{definition}
\newtheorem{defin}[thm]{Definition}

\newtheorem{example}[thm]{Example}
\theoremstyle{remark}

\def\inte{\mathop{\mathrm{int}}}
\def\N{\mathbb{N}}

\begin{document}


\begin{frontmatter}
 
\title{Chaos game for IFSs on topological spaces}
\tnotetext[t1]{Dedicated to the memory of Andrzej Lasota.}
\author{Michael F. Barnsley}
\address{The Australian National University\\
Canberra, Australia}
\ead{mbarnsley@aol.com}
 
\author{Krzysztof Le\'{s}niak}
\address{Faculty of Mathematics and Computer Science, Nicolaus Copernicus University\\
Toru\'{n}, Poland}
\ead{much@mat.umk.pl}

\author{Miroslav Rypka\fnref{ryp}}
\address{Faculty of Science, Palack\'{y} University\\ Olomouc, Czech Republic}
\ead{miroslav.rypka01@upol.cz}

\fntext[ryp]{Supported by the project  StatGIS Team No CZ.1.07/2.3.00/20.0170.}
 
\begin{abstract}
We explore the chaos game for the continuous IFSs on topological spaces.
We prove that the existence of attractor allows us to use the chaos game for visualization of attractor.
The essential role of basin of attraction is also discussed.
\end{abstract}


\end{frontmatter}


\section{Introduction}

A chaos game, defined in Section \ref{secattr}, is a Markov chain Monte
Carlo algorithm applied to describe stationary probability distributions
supported on an attractor of an iterated function system (IFS). Chaos games
can yield efficient approximations to attractors, such as Cantor sets and
Sierpinski triangles, of well-known types of IFS, such as finite sets of
contractive similitudes on $\mathbb{R}^{n}$. They have applications in
computer graphics \cite{nickei} and digital imaging \cite{barnsuper1}.
Rigorous convergence results have been established for chaos games on IFSs
whose maps are contractive on the average, see \cite{Stenflo} and references
therein.

Recently, it has become clear (see \cite{Kameyama}, \cite%
{BarnsleyLesniakRypka}) that attractors of IFSs are of a topological nature.
Applications of the chaos game on a general IFS of continuous maps on a
metric space, without any contractivity conditions on the maps, have been
investigated, see \cite{RPGT, BarnsleyVince-Chaos}), and various somewhat
complicated convergence results have established. In this paper, we simplify
the situation by restricting attention to topological IFSs, defined in
Section \ref{secattr}. Our main result is Theorem \ref{tmain}: under very general
conditions, chaos games yield attractors of topological IFSs.

Our approach is based on understanding the subtle structure of attractors
and their basins of attraction. We proceed as follows. In Section \ref%
{secattr} we give basic definitions and instructive examples of concerning
attractors. In Section \ref{basin}, we discuss the basin of attraction of an
IFS in a topological space and the relationship to a new object, the \textit{%
pointwise basin}; this complements observations made in \cite%
{BarnsleyLesniakRypka}. Theorem \ref{tmain} is proved in Section \ref{chg} and
illustrated by Example \ref{ex:proj}. In Section \ref{conclsec} we relate
this work to the notion of a semi-attractor as defined by Lasota, Myjak and
Szarek.

\section{Attractor of general IFS}

\label{secattr} 
We use the notation and results from \cite{BarnsleyLesniakRypka}.

\begin{defin}
A \textbf{(topological) iterated function system}\textit{\ (IFS)} $\mathcal{W%
}=\{X;w_{1},w_{2},\dots ,w_{N}\}$ is a normal Hausdorff topological space $X$
together with a finite collection of continuous maps $w_{i}:X\rightarrow X$, 
$i=1,2,\dots ,N$. The associated \textbf{Hutchinson operator} $W:\mathcal{K}%
(X)\rightarrow \mathcal{K}(X)$ is defined on the family of nonempty compact
sets $S\in \mathcal{K}(X)$ by $W(S)=\bigcup_{i=1}^{N}w_{i}(S)$. The $k-$fold
composition of $W$ is denoted by $W^{k}$. Similarly we define $W(S)$ for any
set $S\subset X$.
\end{defin}

The Hutchinson operator $W$ on $\mathcal{K}(X)$ and its restriction to $X$, $%
W:X\rightarrow \mathcal{K}(X)$, are continuous, when $\mathcal{K}(X)$ is
endowed with the Vietoris topology (cf. \cite[Proposition 1.5.3 (iv)]{K}).
Without ambiguity we may also write $W^{k}(x)=W^{k}(\{x\})$. Note that $W$
has the order theoretic property $W(S_{1}\cup S_{2})=W(S_{1})\cup W(S_{2})$
for $S_{1},S_{2}\subset X$.

\begin{defin}
\label{chaosdef} A \textbf{chaos game} on the IFS $\mathcal{W=}%
\{X;w_{1},w_{2},\dots ,w_{N}\}$ comprises a sequence of points $\left\{
x_{n}\right\} _{n=0}^{\infty }$ with $x_{0}\in X$ and $x_{n}=f_{\sigma
_{n}}(x_{n-1})$ for $n=1,2,...$ where $\left\{ \sigma _{n}\right\}
_{n=1}^{\infty }$ is a sequence of random variables with values in $%
\{1,2,...,N\}$. (In some cases, but not in this paper, it is required that
these random variables are independent and identically distributed.) In this
paper, in line with realistic applications, we require only that there is a
probability $p>0$ so that for each $n$, $P(\sigma _{n}=m)>p$ for all $m\in
\{1,2,...N\},$ independently of the outcomes for all other values of $n$.
\end{defin}

An IFS may or may not possess some kind of attractor, see below. But, if it
does possess an attractor $A\subset X$, and if a chaos game $\left\{
x_{n}\right\} _{n=0}^{\infty }$ is such that $\overline{\left\{
x_{n}\right\} _{n=0}^{\infty }}=A$, where "$\ \overline{ \cdot}\ $" denotes
topological closure in $X$, then we say "the chaos game yields the
attractor" or some equivalent statement.

\begin{defin}
We call a nonempty compact set $A\subset X$ a (\textbf{strict) attractor} of 
$\mathcal{W}$, when $A$ admits an open neighbourhood $U$ such that $%
W^{k}(S)\rightarrow A$ as $k\rightarrow \infty $, for all nonempty compact $%
S\subset U$. The convergence is understood in the Vietoris sense. The
maximal open set $U$ with the above property is called the \textbf{basin of
(the attractor) }$A$, and is denoted by $B(A)$.
\end{defin}

If $X$ is metrizable with metric $d_{X}$, then it is well-known that the
associated Hausdorff metric, $d_{H(X)}$, induces the Vietoris topology on $%
\mathcal{K}(X)$, \cite{K}.

The question of when the existence of an attractor implies the existence of
a metric with respect to which the IFS is contractive, on a neighborhood of
the attractor, is an active research area, see for example \cite%
{Kameyama,BarnsleyVince-Projective,Vince}.

\begin{example}[\protect\cite{K} Example 4.4.3 (b), \protect\cite%
{BarnsleyVince-Chaos} Example 1]
\label{ex:transitive} Let $X$ be a compactum and $h:X\rightarrow X$ be a
homeomorphism such that $X$ is a minimal invariant set, i.e., if $S\in 
\mathcal{K}(X)$ and $h(S)=S$, then $S=X$. By virtue of the Birkhoff minimal
invariant set theorem (see \cite{Gottschalk,K}) we know that the forward
orbit of any point in $X$ under $h$ is dense in $X$, that is 
\[
\forall _{x_{0}\in X}\;\overline{\{h^{k}(x_{0}):k\geq 0\}}=X.
\]%
A canonical situation of this kind arises for an irrational rotation of the
circle. (Interestingly, a circle is the attractor of a contractive IFS on
the plane, cf. \cite{Circle}.) The IFS $\mathcal{W}=\{X;e,h\}$, where $e$ is
the identity map on $X$, has $A=X$ as a strict attractor with $B(A)=X$.
However, $\mathcal{W}$ is noncontractive and cannot be remetrized into a
contractive system. Moreover, this is an example of an IFS where the
attractor is not point-fibred in the sense of Kieninger (cf. \cite{K}) and
it is not topologically self-similar in the sense of Kameyama (cf. \cite%
{Kameyama}). Thus symbolic techniques, like those in \cite{MauldinUrbanski},
are not directly applicable.

To see that $X$ is the unique strict attractor of $\mathcal{W}$ we observe
the following. First, for all $x_{0}\in X$, 
\[
W^{k}(x_{0})=\{h^{m}(x_{0}):0\leq m\leq k\} \rightarrow\overline{%
\{h^{m}(x_{0}):m\geq0\}}=X. 
\]
Second, for a general $S\in\mathcal{K}(X)$ we have $W^{k}(S)\rightarrow X$.
This is because $W^{k}(x_{0})\subset W^{k}(S)\subset X$ for arbitrary $%
x_{0}\in S$.
\end{example}

Nonmetrizable compact spaces are important for fundamental questions in
measure theory and functional analysis, cf. \cite{Bogachev,Zizler}. However,
Proposition~\ref{th:separability} below shows that it may be cumbersome to
identify a concrete example of a nonmetrizable attractor. Classical
examples, like Tychonoff's product of uncountably many compact factors or
Alexandroff's double circle cannot be used for this purpose, because they
are not separable.

\begin{prop}
\label{th:separability} If $A$ is a strict attractor of a topological IFS,
then $A$ is separable.
\end{prop}

\begin{proof}
Choose any $a_0\in A$. The set
$\bigcup_{k=0}^{\infty} W^k(\{a_0\})\subset A$
is countable and dense in $A$.
\end{proof}

\begin{example}[Non-metrizable attractor]
\label{ex:nonmetric} Let $X=(0,1]\times \{0\}\cup \lbrack 0,1)\times
\{1\}\subset \mathbb{R}^{2}$ be the two arrows space, see for example \cite[%
Example 6.1.20]{Bogachev}. The topology of $X$ is generated by the base of
double intervals 
\[
(x_{0}-r,x_{0}]\times \{0\}\cup (x_{0}-r,x_{0})\times \{1\},
\]%
\[
(x_{1},x_{1}+r)\times \{0\}\cup \lbrack x_{1},x_{1}+r)\times \{1\},
\]%
where $r>0$, $x_{0}\in (0,1]$, $x_{1}\in \lbrack 0,1)$ and double intervals
are tailored to $X$. The space $X$ is known to be compact separable
first-countable and not metrizable. It turns out that $X$ is a strict
attractor of an IFS.

Consider the IFS $\{X;w_{1},w_{2},w_{3}\}$ where 
\[
w_{1}(x,j):=(x/2,j),w_{2}(x,j):=((x+1)/2,j),w_{3}(x,j):=(1-x,1-j),
\]%
for $(x,j)\in X$. One easily observes that $W(X)=X$ and that taking a
counter-image via $w_{i}$ does not change the \textquotedblleft shape" of a
double interval. For instance 
\begin{eqnarray*}
&&w_{1}^{-1}((x_{1},x_{1}+r)\times \{0\}\cup \lbrack x_{1},x_{1}+r)\times
\{1\}) \\
&=&X\cap ((2x_{1},2x_{1}+2r)\times \{0\}\cup \lbrack 2x_{1},2x_{1}+2r)\times
\{1\}),
\end{eqnarray*}%
\begin{eqnarray*}
&&w_{3}^{-1}((x_{0}-r,x_{0}]\times \{0\}\cup (x_{0}-r,x_{0})\times \{1\}) \\
&=&(1-x_{0},1-x_{0}+r)\times \{0\}\cup \lbrack 1-x_{0},1-x_{0}+r)\times
\{1\}.
\end{eqnarray*}%
Therefore $\mathcal{W}=\{X;w_{1},w_{2},w_{3}\}$ constitutes an IFS of
continuous maps. To show that $X$ is a strict attractor of $\mathcal{W}$, it
is enough to verify this on singletons by inspecting the behavior of
cascades parallel to those appearing in the IFS $\{[0,1];f_{1},f_{2}\}$, $%
f_{i}(x)=(x+i-1)/2$, $x\in \lbrack 0,1]$, $i=1,2$. Namely, if $S\subset (0,1)
$ is dense in $(0,1)$ w.r.t. the Euclidean topology, then $S\times \{0,1\}$
is dense in $X$.
\end{example}

\section{Basins and pointwise basins of attraction}

\label{basin}

The basin of attraction $B(A)$ plays a key role in chaos games on a general
IFS, because it is usually required that $x_{0}$ in Definition \ref{chaosdef}
belongs to $B(A)$, to permit the chaos game to yield the attractor $A$. In
fact, it is only necessary that $x_{0}$ belongs to the pointwise basin of $A$%
, defined next. In this section we examine properties of basins and
pointwise basins.

\begin{defin}
The \textbf{pointwise basin} of a set $\widetilde{A}\subset X$ (w.r.t. the
IFS $\mathcal{W}$) is defined to be 
\[
B_{1}(\widetilde{A})=\{x\in X:W^{k}(\{x\})\rightarrow \widetilde{A}\}.
\]%
Note that $B_{1}(\widetilde{A})$ may be empty. If $\mathop{\mathrm{int}}%
B_{1}(\widetilde{A})\supset $ $\widetilde{A}$, then we say that $\widetilde{A%
}$ is a \textbf{pointwise (strict) attractor }of $\mathcal{W}$.
\end{defin}

\begin{prop}
\label{basprop} Let $A$ be a strict attractor of $\mathcal{W}$ with basin $%
B(A)$. Then $A$ is a pointwise strict attractor of $\mathcal{W}$ and $%
B(A)=B_{1}(A)$.
\end{prop}

\begin{proof}
We shall check only that $B_1(A)\subset B(A)$, because the reverse
inclusion is obvious. Fix $S\in\mathcal{K}(B_1(A))$. For every $x\in S$,
$W^{k}(\{x\})\to A$. Since $B(A) \supset A$ is an open neighbourhood,
there exists $k(x)$ such that $W^{k(x)}(\{x\}) \subset B(A)$. Being each
$W^{k(x)}: X \rightarrow \mathcal{K}(X)$ continuous,
we can find open neighbourhoods $U_x\ni x$ so that
$W^{k(x)}(U_x) \subset B(A)$. 

Next, we take finite subcovering
$\{U_{x_j}\}_{j=1}^{m}$ of $\{U_x\}_{x\in S}$.
By normality of $X$ we can divide $S= \bigcup_{j=1}^{m} S_{j}$
in such a way that $S_{j} \in\mathcal{K}(U_{x_j})$;
cf. \cite{BarnsleyLesniakRypka}. For each separate $j$ one has
$W^{k}(S_j) \to A$, so there exists $k_j$ with the property
\begin{equation}\label{eq:jtersuBA}
W^{k}(S_j) \subset B(A)
\end{equation}
for $k\geq k_{j}$. Hence (\ref{eq:jtersuBA}) holds for
$k\geq k_{0} := \max_{j=1,\dots, m} k_j$. 
Since $W^{k_0}$ is continuous, there exists an open neighbourhood
$U(S)$ of $S$ which is also mapped with $W^{k_0}$ into $B(A)$. 
Being $S$ arbitrary and $U(S)$ open, we can conclude that $B_1(A)\subset
B(A)$.
\end{proof}

The following example illustrates a pointwise strict attractor $\widetilde{A}
$ that is not an attractor.

\begin{example}
\label{exDK} Consider the set of points $X$ on the circle $C$ which may be
projected to integers and infinity on real line $\mathbb{N}^{\ast }$. Let $f:%
\mathbb{N}^{\ast }\rightarrow \mathbb{N}^{\ast }$ be such that $%
f(x)=x+1,x\neq \infty $ and $f(\infty )=\infty $ (see Figure \ref{countex}).
Observe that the map is continuous with respect to the Euclidean metric on
the circle. It is obvious that each point of $X$ is attracted to the north
pole by the induced mapping on $X,$ so we have $B_{1}(\infty )=X$. However, $%
\{\infty \}$ is not an attractor of the IFS $\{X;f\}$.
\end{example}


\begin{figure}[ptb]
%
\centerline{\resizebox{10cm}{!}{\includegraphics[trim = 3cm 4cm 3cm 4cm,
clip]{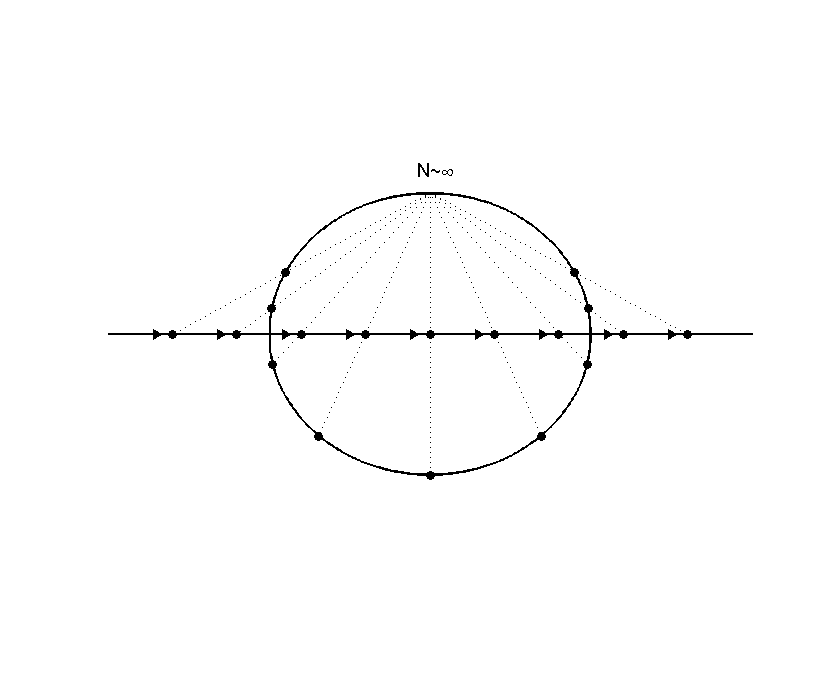}}}
\caption{Counterexample exhibiting that an open pointwise basin need not be
a basin.}
\label{countex}
\end{figure}

A positive criterion for the existence of a strict attractor $A$ according
to the non\-empti\-ness of $B_1(A)$ reads as follows.

\begin{lem}
\label{th:pointwiselem} Let $(X,d)$ be a metric space and $w_i: X\to X$, $%
i=1,\dots,N$ nonexpansive mappings, i.e., 
\[
d(w_i(x_1), w_i(x_2)) \leq d(x_1,x_2) \mbox{ for } x_1,x_2\in X. 
\]
Let $A\in \mathcal{K}(X)$. If $\mathop{\mathrm{int}} B_1(A) \supset A$, then 
$A$ is a strict attractor of $\{X; w_1,\dots,w_N\}$ with the basin $B(A) =
B_1(A)$.
\end{lem}

\begin{proof}
Fix ${\varepsilon} > 0$ and a nonempty compact $S\subset B_1(A)$.
Choose a finite net $S_{\varepsilon}\subset S$,
$d_{H}(S_{\varepsilon},S)<\varepsilon$. From the additivity of $W$ it follows
that $W^{k}(S_{\varepsilon}) \to A$, because
$S_{\varepsilon}$ is a finite subset of $B_1(A)$.
Then $d_{H}(W^{k}(S_{\varepsilon}),A) < \varepsilon$ for large enough $k$.
From nonexpansiveness,
$d_{H}(W^{k}(S), W^{k}(S_{\varepsilon})) \leq d_{H}(S,S_{\varepsilon})$.
Hence we get $d_{H}(W^{k}(S),A) < 2\varepsilon$ for large $k$.
Therefore
\begin{equation}\label{eq:attractor}
W^{k}(S)\to A \mbox{ for compact } S\subset B_1(A).
\end{equation}
\end{proof}

\begin{prop}
\label{basinv} If $\mathop{\mathrm{int}} B_1(A) \supset A$, then the
pointwise basin is open and positively invariant:

\begin{enumerate}
\item[(i)] $\mathop{\mathrm{int}} B_1(A) = B_1(A)$,

\item[(ii)] $W(B_1(A))\subset B_1(A)$.
\end{enumerate}

In particular, the basin of an attractor is positively invariant, $%
W(B(A))\subset B(A)$.
\end{prop}

\begin{proof}
Ad (i). Denote $V = \inte B_1(A)$. Fix $x\in B_1(A)$.
There exists $k_0$ such that $W^{k_0}(x)\subset V$.
By continuity of $W^{k_0}$
we can find an open neighbourhood $U_{x}$ of $x$ satisfying
$W^{k_0}(U_{x})\subset V$. We shall check that $U_{x}\subset B_1(A)$.
 
Since for each $y\in U_{x}$ the set $W^{k_0}(y)\subset V\subset B_1(A)$
is finite, we get $W^{k}(y) = W^{k-k_{0}}(W^{k_0}(y)) \to A$. Therefore
$x$ belongs to $B_1(A)$ together with its neighbour points $y$.
 
Ad (ii). Let $y\in W(x)$, $x\in B_1(A)$. We will show that $y\in B_1(A)$.
 
Obviously, given an open $V\supset A$, we have that
$W^{k}(y) \subset W^{k+1}(x) \subset V$ for large $k$.
In particular,
\[
W^{k_0}(y)\subset W^{k_{0}+1}(x)\subset \inte B_1(A)
\]
for some $k_0$.
 
Now, fix an open $V\cap A\neq \emptyset$. Pick anyhow $z\in W^{k_{0}}(y)$.
Since $z\in B_1(A)$, we have
\[
W^{k+k_{0}}(y) \cap V \supset W^{k}(z)\cap V\neq\emptyset
\]
for large $k$. Altogether, $W^{k}(y)\to A$. Therefore $y\in B_1(A)$.
 
Invariance of $B(A)$ can be inferred from the above by applying Proposition \ref{basprop}.
\end{proof}

In general $W(B_1(A)) \neq B_1(A)$; just take a projection onto a point.

Finally, we show how important is to iterate compact sets, not merely
(closed) bounded sets like it was the case in the original paper \cite%
{Hutchinson} and its successors.

\begin{example}[\protect\cite{AFGL} Example 2]
Let $X= {\ell}^2$ be the Hilbert space of square summable sequences with an
orthonormal basis $e_1, e_2,\ldots$. We define $w:X\to X$ via 
\[
w(x) = \sum_{i=1}^{\infty} \left(1-\frac{1}{i+1}\right)\cdot \lambda_{i} e_i 
\]
for $x=\sum_{i=1}^{\infty} \lambda_{i} e_{i}$. One easily checks that $w$ is
nonexpansive. We shall prove that $A= \{0\}$ is a strict attractor of the
IFS $\{X; w\}$ with the full basin of attraction $B(A)=X$, but it does not
attract iterates of nonempty closed bounded noncompact sets, no matter how
small these sets are.

First, we examine the convergence 
\begin{equation}  \label{eq:Edelstein}
w^k(x) \to 0 \;\mbox{ where }\; x=\sum_{i=1}^{\infty} \lambda_{i} e_{i},
k\to\infty.
\end{equation}
Obviously 
\[
w^k(x) = \sum_{i=1}^{\infty} \left(1-\frac{1}{i+1}\right)^{k} \cdot
\lambda_{i} e_i. 
\]
We need 
\[
\|w^k(x)\|^2 = \sum_{i=1}^{\infty} \left(1-\frac{1}{i+1}\right)^{2k} \cdot
\lambda_{i}^2 \to 0. 
\]
To verify this we employ the Lebesgue dominated convergence theorem adapted
to series (understood as integrals with respect to the counting measure).
Indeed both, the pointwise convergence 
\[
\left(1-\frac{1}{i+1}\right)^{2k}\cdot \lambda_{i}^2 \to 0,\; i=1,2,\dots 
\]
and the majorization 
\[
\sum_{i=1}^{\infty} \left(1-\frac{1}{i+1}\right)^{2k}\cdot \lambda_{i}^2
\leq \sum_{i=1}^{\infty} \lambda_{i}^2 < \infty,\; k=1,2,\dots, 
\]
hold true. Thus $X=B_1(A)$.

Second, we bootstrap (\ref{eq:Edelstein}) to (\ref{eq:attractor}) via Lemma %
\ref{th:pointwiselem}. Thus $A=\{0\}$ is a strict attractor.

Third, we note that 
\[
\|w^k(x_k)-w^k(0)\| = \left(1-\frac{1}{k+1}\right)^{k} \cdot r\to r\cdot
\exp(-1)>0, 
\]
where $x_k=r\cdot e_k$, which is the element of the sphere $D_r = \{x: \|x\|
=r\}$ at $0$ with radius $r$. Hence $A$ does not attract closed bounded
sets; $\overline{W^k(D_r)} \not\rightarrow A$, for arbitrarily small radii $%
r $.
\end{example}

\section{Chaos game}

\label{chg}

We are ready to show that the chaos game works on topological spaces in a
rather general framework. The initial point of the orbit must be picked from
the pointwise basin of attraction.

\begin{thm}
\label{tmain} Let $\mathcal{W}=\{X; w_1,w_2,\dots,w_N\}$ be an IFS. Let $A$
be a nonempty compact set such that $\mathop{\mathrm{int}} B_1(A)\supset A$.
If $\{x_k\}_{k=0}^{\infty}$ is a random orbit under $W$ starting at $x_0\in
B_1(A)$, then with probability one, 
\[
A=\lim_{K\to\infty} \overline{\{x_k\}_{k=K}^{\infty}}, 
\]
where the limit is taken with respect to the Vietoris topology and each
mapping $w_{i}$ is chosen at least with probability $p>0$.
\end{thm}

\begin{proof}
Fix $x_0\in B_1(A)$. We have
\begin{equation}\label{eq:LimA}
W^{k}(x_0) \to A.
\end{equation}
 
First, for every open set $V\supset A$ there exists $j_0$ such that
\[
\overline{\{x_j: j\geq j_1\}} \subset \overline{\bigcup_{j\geq j_1} W^{j}(x_0)} \subset V
\]
for all $j_1\geq j_0$. This is due to (\ref{eq:LimA}).
 
Second, we will show that, for any open $V\cap A\neq \emptyset$ and all $k_0$,
$\{x_k\}_{k=k_0}^{\infty} \cap V\neq \emptyset$ with probability one.
Let us denote
\[
K = A\cup \bigcup_{k\in \N}W^k(x_0)=\overline{\bigcup_{k\in \N}W^k(x_0)}.
\]
By (\ref{eq:LimA}) the set $K$ is compact. Moreover, $K\subset B_1(A)$, thanks to Proposition \ref{basinv}.
 
Fix an open $V\cap A\neq \emptyset$. To each $x\in K$
we can assign $k(x)$ in such a way that $W^{k(x)}(\{x\}) \cap V\neq \emptyset$.
This is possible, because (\ref{eq:LimA}) holds for $x_0$ replaced with $x\in B_1(A)$.
 
Now we can use the continuity of $W^{k(x)}: X\to \mathcal{K}(X)$.
For each $x\in K$ there exists an open $U_x\ni x$ such that
$W^{k(x)}(y)\cap V\neq \emptyset$ for all  $y\in U_x$.
The open covering $\{U_x\}_{x\in K}$ of $K$ admits a finite subcovering $\{U_{x_i}\}_{i=1}^{m}$.
We put $k_{*} = \max \{k(x_i) : i=1,\dots,m\}$. Therefore for every $x\in K$ there exists
$0\leq k \leq k_{*}$ such that $W^{k}(x)\cap V\neq\emptyset$.
Hence there exists a finite word $\sigma_{1}\dots\sigma_{k}\in \{1,\dots,N\}^{k}$
(of the length not exceeding $k_{*}$) which satisfies
$w_{\sigma_{k}}\circ\dots w_{\sigma_{1}}(x) \in V$.
 
Each map $w_{\sigma_{j}}$, $1\leq j\leq k\leq k_{*}$, is drawn with the probability
not less than $p$. Having drawn the point $x_{L}$, $L\geq 0$, the probability that
$x_{L+k}\in V$ is not less than $p^{k_{*}}$.
 
Denote by $E_n$ the event such that $x_k\in V$ for some $k\leq n$.
The complementary event shall be written as $E_n^{c}$.
Taking into account the observations made so far, basic conditional probability calculation shows that
\[
P(E_{(n+1)k_{*}}\,|\,E_{n k_{*}}^{c}) \geq p^{k_{*}}.
\]
Thus $\sum_{n=1}^{\infty} P(E_{(n+1)k_{*}}\,|\,E_{n k_{*}}^{c}) = \infty$.
Moreover $E_{n}\subset E_{n+1}$.
On calling the second Borel--Cantelli lemma (\cite[chap.1 p.18]{Chandra}), we get that,
with probability $1$, $E_{nk_{*}}$ happens for infinitely many $n$.
Overall we are almost sure that all tails $\{x_k\}_{k=k_0}^{\infty}$ intersect
those open sets $V$ which are intersecting an attractor.
\end{proof}

A class of noncontractive IFSs with an attractor may be found in projective
spaces.

\begin{example}
\label{ex:proj} (\cite[Example 3]{BarnsleyVince-Projective}) Consider the
IFS $F = \{\mathbb{P}^2; f_1,f_2\}$, where 
\[
\begin{array}{cc}
f_1 =\left(%
\begin{array}{ccc}
41 & -19 & 19 \\ 
-19 & 41 & 19 \\ 
19 & 19 & 41 \\ 
\end{array}
\right), \  & \ f_2 = \left(%
\begin{array}{ccc}
-10 & -1 & 19 \\ 
-10 & 21 & 1 \\ 
10 & 10 & 10 \\ 
\end{array}
\right).%
\end{array}
\]
Neither of the functions has an attractor. However, the IFS consisting of
both functions has an attractor plotted in Figure \ref{ci}. It is obtained
by means of the chaos game.
\end{example}

\begin{figure}[ptb]
%
\centerline{\resizebox{10cm}{!}{\includegraphics[trim = 1cm 1cm 1cm 1cm,
clip]{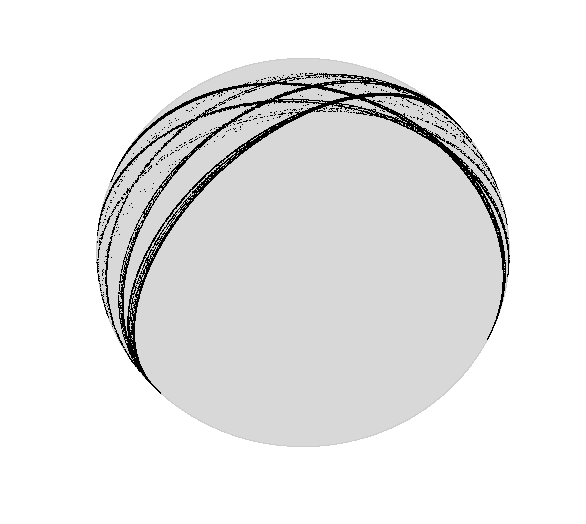}}}
\caption{Chaos game for Example \protect\ref{ex:proj}}
\label{ci}
\end{figure}

\section{\label{conclsec}Concluding remarks}

One can still weaken the uniform positive minorization of drawing
probabilities and allow for some decay in time (say logarithmic and alike).
We refer to \cite{Lesniak} for these more involved nonstationary conditions.

Careful examination of the above proof yields

\begin{thm}[chaos game for a Lasota-Myjak semiattractor]
\label{th:semiattractor} Let $X$ be a Hausdorff topological space. Let $A$
be a compact semiattractor of the IFS $\{X; w_1,w_2,\dots,w_N\}$ comprising
continuous maps. If $\{x_k\}_{k=0}^{\infty}$ is a random orbit under $W$
starting at $x_0\in A$, then with probability one, 
\[
A=\overline{\{x_k: {k=0,1,2,\dots}\}}, 
\]
provided each mapping $w_{i}$ is chosen at least with probability $p>0$.
\end{thm}

We refer to \cite{LasotaMyjak, LasotaMyjakSzarek} for the notion of
semiattractor.

A number of remarks concerning the above two theorems is in order.

\begin{enumerate}
\item The proof of Theorem \ref{tmain} relies on compactness of $A$ and
continuity of $W$. It works actually in Hausdorff nonnormal spaces too,
although preparatory material from Section \ref{basin} involved normality in
few places.

\item Lasota-Myjak semiattractors are defined for multivalued iterated
function systems so some discontinuity of maps is allowed. Moreover,
semiattractor can be noncompact.

\item We restrict the chaos game to a concrete class of stochastic processes
which randomly draw the maps from a system. Thanks to this we do not have to
validate whether the Markov operator associated with the IFS is
asymptotically stable; in particular, the IFS at our disposal need not
fulfill the standard average contraction condition (cf. \cite%
{LasotaMyjakSzarek}).

\item As already pointed out in \cite{BarnsleyVince-Projective}, the nature
of noncontractive IFSs possessing attractors is still not well understood.
\end{enumerate}

The authors would like to thank Dominik Kwietniak for suggesting Example \ref%
{exDK}. 

\end{document}